\documentclass[a4paper]{article}
\usepackage{amsmath,amssymb,amsfonts}
\usepackage{color}
\usepackage{graphicx,psfrag}
\usepackage{tikz}
\usepackage{amsthm}

\newtheorem{theorem}{Theorem}
\newtheorem{lemma}{Lemma}
\newtheorem{corollary}{Corollary}
\newtheorem{proposition}{Proposition}
\newtheorem{rem}{Remark}

\newtheorem{prob}{Problem}

\newcommand{\until}[1]{\{1,\dots, #1\}}

\newcommand{\supscr}[2]{#1^{\textup{#2}}}

\newcommand{\eps}{\varepsilon}

\newcommand{\abs}[1]{\left|#1\right|}
\newcommand{\prt}[1]{\left(#1\right)}

\newcommand{\real}{\mathbb{R}}

\renewcommand{\natural}{\mathbb{N}}

\newcommand{\xeq}{\supscr{x}{eq}}
\newcommand{\xrand}{\supscr{x}{rand}}
\newcommand{\xrandm}{\supscr{x}{rand,$m$}}

\newcommand{\E}{\mathbb{E}} 
\providecommand{\Pr}{\mathbb{P}}

\title{Optimal one-dimensional coverage by unreliable sensors}

\author{Paolo Frasca\thanks{P. Frasca is with Department of Applied Mathematics, University of Twente, 7500 AE Enschede, The Netherlands
{\tt\small p.frasca@utwente.nl}.}%
\and Federica Garin\thanks{F. Garin is with NeCS team, INRIA Grenoble -- Rh\^one-Alpes, France        {\tt\small federica.garin@inria.fr
}.}%
\and Bal\'azs Gerencs\'er\thanks{B. Gerencs\'er and J. M. Hendrickx are with ICTEAM Institute,
  Universit\'e Catholique de Louvain, Belgium
  {\tt\small balazs.gerencser@uclouvain.be} and {\tt\small
    julien.hendrickx@uclouvain.be} 
Their work is supported by the DYSCO Network (Dynamical Systems,
Control, and Optimization), funded by the Interuniversity
Attraction Poles Programme, initiated by the  Belgian
Federal Science Policy Office, and by the Concerted Research Action (ARC) of the
French Community of Belgium.}%
\and Julien M. Hendrickx\footnotemark[3]
}

\begin{document}
\maketitle

\begin{abstract}
This paper regards the problem of optimally placing unreliable sensors in a one-dimensional environment. We assume that sensors can fail with a certain probability and we minimize the expected maximum distance from any point in the environment to the closest active sensor. We provide a computational method to find the optimal placement and we estimate the relative quality of equispaced and random placements. We prove that the former is asymptotically equivalent to the optimal placement when the number of sensors goes to infinity, with a cost ratio converging to 1, while the cost of the latter remains strictly larger.
\end{abstract}





\section{Introduction}

Sensor networks are used to monitor large or hazardous environments, for purposes ranging from oceanographic research to security in airports, industrial plants, and other complex infrastructures.
In order to provide the best {\em coverage} of the assigned environment, sensors have to be deployed at suitable locations. 
As sensors are prone to failures in collecting and transmitting data, the robustness of the obtained coverage performance is a natural concern: thus, we consider in this paper the problem of placing {\em unreliable} sensors in a given environment in order to provide the optimal coverage of it.

Coverage optimization and related problems of optimal facility location have been studied by the operations research community for a long time, often using concepts from geometric optimization and computational geometry~\cite{FPP-MIS:93,QD-VF-MG:99}. 
During the past decade, conditions for sensor networks to provide a certain level of coverage have been found in a variety of situations, which include both random and deterministic placement strategies~\cite{AG-SKD:08}. Many available results allow sensors to fail or to spend time in a sleeping mode to save energy: in fact, these two scenarios can be given a unified treatment~\cite{SK-THL-JB:04,SS-RS-NS:05} using probabilistic methods~\cite{PH:88,MF-RM:07}. However, it appears that the issue of the optimality of such placements, although recognized as central, has been left in the background~\cite{MY-KA:08}.

Control scientists have also become interested in these topics, after realizing that  feedback control can enable the autonomous deployment of self-propelled sensors~\cite{AH-MM-GS:02}. The main references for this research are the book~\cite{FB-JC-SM:09} and the related papers~\cite{JC-SM-TK-FB:02j,JC-FB:02m}, while very recent developments include~\cite{FB-RC-PF:08u,JRM-AW:13,JLN-GJP:13,NEL-AO:13}.
Most literature from the control community assumes sensors to behave reliably, but recent results are making clear that this assumption is not free from risks. In fact, sensor failures deteriorate the performance of the sensor network and it is not even clear if optimal solutions derived for the case without failures retain good properties in other cases. Indeed, simulations reported in~\cite{SH-TB:12} show the solutions that are optimal in the presence of failure are qualitatively different from those optimal in the fully reliable case.

The common sense countermeasure to failing sensors is adding some redundancy and letting more than one sensor ``responsible'' for covering a certain region of the environment, so that they can back up each other in case a failure occurs. To this aim, sensors can cluster into groups, such that the members of each group have the same location. This approach has been exploited by Cort\'es~\cite{JC:12}, under the assumption that the number of failed sensors is precisely {\em known}. As a consequence, the number of clusters in the optimal solution is directly determined by this number. 

In this paper, we consider the problem of {\em optimal disk-coverage in a one-dimensional environment by unreliable sensors}, under a probabilistic failure model that does not assume any {\it a priori} information about the number or the location of the failures. Rather, we assume that sensors fail independently and with the same probability. We then aim to minimize, in expectation, the largest distance between a point in the environment and an active sensor.

This cost function was already used in \cite{JC:12}, which was motivated by random field estimation~\cite{RG-JC:07}. It is consistent with the spirit of standard coverage questions in sensor networks, in which one is interested in guaranteeing a full coverage of the environment using a given number of sensors with a certain coverage radius~\cite{SK-THL-JB:04,HP:06}. Note that it also corresponds to the classical problem of facility location, where a number of facilities have to service customers in a given area and want to optimize the worst-case servicing delay~\cite[Ch.~2]{FB-JC-SM:09}.

Regarding the choice of the environment, most prior works about sensor networks have chosen two-dimensional settings. In contrast, our choice of working in dimension one allows us to achieve sharper characterizations and results about optimality, both asymptotical and for finite networks. Results of this kind are scarce in the literature, even if one-dimensional settings have often been studied, both in classical~\cite{PH:88} and recent works~\cite{NEL-AO:13,LL-BZ-JZ:13}.

Our first result -- Theorem~\ref{prop:linprog} -- states that the problem at hand is equivalent to a linear program, albeit with a number of variables growing exponentially with the number of sensors. This fact allows for a computational solution that is tractable if the number of sensors is not large. Secondly, we show that for large number of sensors $n$, the cost of the equispaced placement decreases to zero with leading term $\frac1{2\log{p^{-1}}}\frac{\log{n}}{n}$, where $p$ is the probability of failure. In Theorem~\ref{th:equispace-logn}, we provide analytic bounds on the optimal cost and prove that the equispaced placement is nearly optimal: the ratio between its cost and the optimal cost tend to 1 when $n$ grows. By
contrast, we show in Theorem~\ref{thm:random} that a random placement has a larger cost of order $\frac{1}{2(1-p)} \frac{\log n}{n}$. The almost optimality of the deterministic placement and its strict difference with the random placement had not been noticed before in the literature.

Our analysis also bear consequences for the failure model adopted by Cort\'es~\cite{JC:12}: for instance, we show that the equispaced  placement is nearly optimal in this case as well. Finally, we note that our results extend and refine those recently presented by some of the authors in~\cite{PF-FG:13}, where a similar model of unreliable coverage was proposed. 

\subsection*{Paper structure}
The rest of the paper is organized as follows. The formal definition of the
problem is presented in \S~\ref{sect:prob_def}. Translation
to a linear optimization problem is shown in
\S~\ref{sect:linprog}. In \S~\ref{sect:perf_equi} we
assess the performance of the equispaced placement. In
\S~\ref{sect:extreme_p} we analyze the special cases when the failure probabilities are close to 0 or to 1. \S~\ref{sect:random} deals with
the case of random sensor placement. In \S~\ref{sect:cortes} we
adapt our results to the failure model by Cort\'es. Conclusions are drawn in \S~\ref{sect:conclusion}.

\section{Problem definition}\label{sect:prob_def}%
We assume that we have a set of sensors indexed in $[n]=\until{n}$ which have to cover the interval $[0,1]$. Since sensors may fail, we consider for each placement $x\in [0,1]^n$ the coverage cost defined as the largest distance between a point in $[0,1]$ and its closest {\em active} (not failing) sensor.
To formalize this notion, we let $A$ denote the set of active sensors: we will use $\abs{A}$ to denote the cardinality of $A$ and $A_k$ to denote the $\supscr{k}{th}$ smallest index present in the set $A$, for $k=1,\dots,|A|$. We also call $x_A\in [0,1]^{|A|}$ the restriction of the vector $x$ to those entries for which the corresponding sensors are active. The cost incurred when the set of sensors $A$ is active is thus
\begin{equation}\label{eq:def_C0_xA}
C_0(x_A) = \max_{s \in [0,1]} \min_{j \in A} | s - x_j |.
\end{equation}
To be formally complete, we assign the arbitrary cost
$C_0(x_{\emptyset}) = 1$ to the situation where all sensors fail. This
convention has no effect when we seek to optimize the locations of the
sensors, as locations are irrelevant when they all fail. 
Observe that if no sensor fails ($A=[n]$), then the cost~\eqref{eq:def_C0_xA} reduces to
\[ C_0(x) = \max_{s \in [0,1]} \min_{j \in [n]} | s - x_j |. \]
In this case, it is known that the equispaced placement of $n$ sensors, namely
\begin{equation}\label{eq:defxeq}
\xeq=\frac1{2n}(1, 3, \dots, 2n-1),
\end{equation}
is the optimal solution and achieves a cost $C_0(\xeq)=\frac1{2n}$.
Since we assume that failures are random, we define the event $E_A = \{ \text{$A$ is the set of active sensors} \}$ and we consider the {\em expected value} of the cost $C_0$, which is 
\begin{equation}\label{eq:def-average-cost}
C(x) =   \sum_{A \subseteq{[n]} } \Pr(E_A) C_0(x_A),
\end{equation}
where $\Pr(E_A)$ is the probability of $E_A$.
In the rest of this paper, with the exception of \S~\ref{sect:cortes}, we assume that each sensor fails with probability $p$, independently from the others. Consequently, 
\begin{equation}
\label{eq:prob-A-p}\Pr(E_A)=p^{n-\abs{A}}(1-p)^{\abs{A}}. \end{equation}
We are then ready to formally state our optimization problem.
\begin{prob}[Independent failures]\label{prob:mainprob}
For given $p\in (0,1)$ and $n\in \natural$, find $x^*\in [0,1]^n$ that minimizes
the cost~\eqref{eq:def-average-cost} with~\eqref{eq:prob-A-p}.
\end{prob}

In what follows we assume, for simplicity and without loosing generality, that $x$ is ordered $x_1 \le x_2 \le \dots\le x_n$. This assumption implies that 
\begin{equation}\label{eq:sorted_C0_A}
C_0(x_A)=  \max \left\{ x_{A_1}, 1-x_{A_{\abs{A}}},  \max_{k=1, \dots, \abs{A} -1} \frac{1}{2}(x_{A_{k+1}}-x_{A_k})\right\}.
\end{equation}

\section{Formulation as a linear program}\label{sect:linprog}

A solution of Problem~\ref{prob:mainprob} can be numerically computed by means 
of the following result,  that shows its equivalence to a suitable linear program.

\begin{theorem}[Linear program]\label{prop:linprog}
Let $n\in \natural$ and $p\in (0,1)$. The (ordered) vector $x^*\in [0,1]^n$ is an optimal solution of Problem~\ref{prob:mainprob} if and only if there exists a vector $w^*\in \real^{{2^n}-1} $ such that $(x^*,w^*)$ is an optimal solution to the following linear program:
\begin{align}
\min &\sum_{A\neq\emptyset} \Pr(E_A)w_A 
\label{obj_in_thm_linprog}
\\
s.t. \nonumber\\
 &0\leq x_1\leq \dots \leq x_n\leq 1, \label{constr:sorted_in_thm_linprog}\\
\text{and } &\forall\, A \subseteq [n],~A\neq\emptyset,\nonumber\\
&w_A \geq \frac{1}{2}(x_{A_{k+1}}- x_{A_k}), \hspace{.2cm} \text{for } k =1,\dots,|A|-1, \label{constr:def_cost_normal_in_thm_linprog}\\
&w_A \geq x_{A_1},w_A \geq 1-x_{A_{|A|}}.\label{constr:def_cost_boundary_in_thm_linprog}
\end{align}
\end{theorem}

\begin{proof}
As the constant term $\Pr(E_{\emptyset})$ can be ignored when looking
for the $x$ minimizing $C(x)$, Problem~\ref{prob:mainprob} is
equivalent to $$\min_{x_1\leq \dots \leq x_n}\sum_{A \subseteq [n],A\neq\emptyset} \Pr(E_A) C_0(x_A).$$ Since $\Pr(E_A)\geq 0$ for every $A$, this problem is in turn equivalent to 
$$
\min_{x_1\leq \dots \leq x_n}\sum_{A \subseteq [n],A\neq\emptyset} \Pr(E_A) w_A \text{     s.t.   } w_A \geq C_0(x_A) \text {   for every   } A\neq \emptyset,
$$
that is, to~\eqref{obj_in_thm_linprog} under the constraints~\eqref{constr:sorted_in_thm_linprog} and $w_A \geq C_0(x_A)$ for every $A\neq \emptyset$.
Thanks to~\eqref{eq:sorted_C0_A}, the constraint $w_A \geq C_0(x_A)$ can be separated in 
$w_A \geq x_{A_1}$, $w_A \geq (1-x_{A_{\abs{A}}})$, and $w_A \geq \frac{1}{2} ( 
x_{A_{k+1}}-x_{A_k}
)$ for $k=1\dots,\abs{A}-1$, that is, in~\eqref{constr:def_cost_normal_in_thm_linprog} and~\eqref{constr:def_cost_boundary_in_thm_linprog}, which achieves our proof.
\end{proof}

\medskip

The formulation as a linear program implies that the optimal solution corresponds to one of the vertices of the polytope defined by the constraints. Unfortunately, the number of such constraints is exponentially large in the number of sensors and thus the program becomes quickly intractable.
Nevertheless, we are able to calculate the optimal placements as long as $n$ is not too large. In Figure~\ref{fig:sim12} we
illustrate the evolution of the optimal placement for Problem~\ref{prob:mainprob} as a function of $p$.
We can see that the dependence on $p$ is rather complex and it is not clear how, or if, one could provide a simple exact description of the optimal location of the sensors as a function of $n$ and $p$. 
 Still, in \S~\ref{sect:extreme_p} we will show that the
  equispaced placement is optimal when $p$ is near 0 and a single
  cluster at $1/2$ is optimal when $p$ is near 1.

Observe that the the optimal $x$ is a piecewise constant function of $p$. This feature can actually be explained by the structure of the linear program in Theorem~\ref{prop:linprog}. Indeed, one can see that the constraints do not depend on $p$, which only affects the cost function. 
For any $p$, one can thus always find an optimal $(x^*,w^*)$ among the finitely many vertices of the polytope defined by these constraints. It is therefore natural to observe only finitely many different optimal solutions.

\begin{figure}[ht]
\psfrag{x}{$x$}
\psfrag{p}[][][1][0]{$p$}
\centering
\includegraphics[width=0.8\columnwidth]{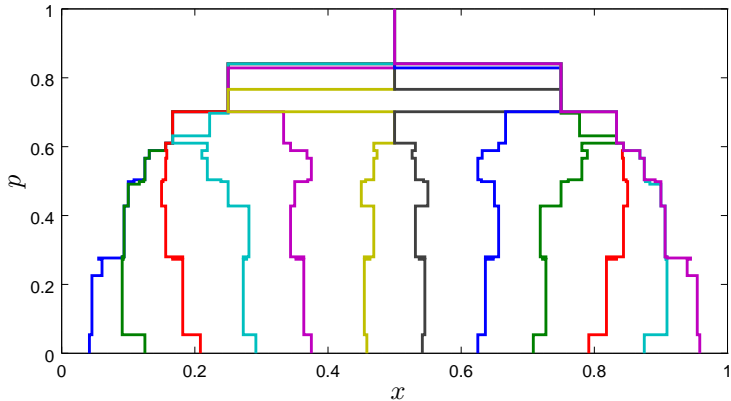}
\caption{Optimal sensor placement for Problem~\ref{prob:mainprob} for $n=12$ sensors and varying $p$.}
\label{fig:sim12}
\end{figure}

\section{Performance of the equispaced placement}\label{sect:perf_equi}
\label{sec:perf_equidist}
The difficulty of providing explicit formulas or efficient computational methods to solve Problem~\ref{prob:mainprob} motivates us to investigate the properties of simple near-optimal solutions. We concentrate on the equispaced placement, which we have seen to be optimal in case of no failures, achieving a cost $C_0(\xeq)=\frac1{2n}$. In the case of positive failure probability, we can prove that the cost of the equispaced placement is nearly optimal. 
\smallskip

\begin{theorem}[Cost of equispaced]\label{th:equispace-logn}
Let $p\in (0,1)$ and let $x^*$ denote the optimal placement for this $p$. Then,
\begin{equation}
 C(\xeq)=\frac1{2\log{p^{-1}}}\frac{\log{n}}{n} + O\!\left(\frac{1}{n}\right) \qquad \text{for $n\to\infty$}
\label{eq:exp-equally-theorem_simple} \end{equation}
and for every $n\in \natural$
\begin{equation}
\label{eq:almost-optimal-bound}
C(\xeq)\le C(x^*) + \frac{p}{1-p} \frac2{n}. 
\end{equation}
\end{theorem}

\smallskip
Equation~\eqref{eq:exp-equally-theorem_simple} is illustrated in Figure~\ref{fig:rand-vs-equi} (\S~\ref{sect:random}).
A few relevant observations follow from this theorem:  (i) the order of growth of $C(\xeq)$ is only worse than the order of $C_0(\xeq)$ by a logarithmic factor; (ii) $\xeq$ asymptotically achieves the optimal cost, since $\frac{C(\xeq)}{C(x^*)}\to 1$; and (iii) the difference in cost between $\xeq$ and the optimum can be estimated at finite $n$, too. Consequently, the equispaced placement can be seen as a valid heuristic solution, when finding an exact solution proves to be intractable.

The rest of this section is devoted to prove Theorem~\ref{th:equispace-logn}.  We first prove equation~\eqref{eq:exp-equally-theorem_simple} in \S~\ref{sect-proof-paolo}: its proof is based on classical results about the properties of the runs of consecutive ones in sequences of Bernoulli trials.
Next, in \S~\ref{sec:line-circle} we  prove~\eqref{eq:almost-optimal-bound}; the proof of this formula relies on an alternative version of Problem~\ref{prob:mainprob} defined on the circle, for which the equispaced solution is actually optimal.

\subsection{Longest runs of failures and proof of~\eqref{eq:exp-equally-theorem_simple}}
\label{sect-proof-paolo}

Let $R_n$ be the maximum number of sensors which fail ``in a row'', {\it i.e.}, the length of the longest run of failures over $n$ sensors.
The random variable $R_n$ is closely related to the cost, as we detail below. 
On the other hand, the distribution of $R_n$ and its asymptotic behavior for large $n$ are well studied in the literature,
due to their relevance in combinatorics~\cite{PE-PR:75}. The following lemma, taken from~\cite{LG-MS-MW:86}, characterizes the asymptotic behavior of $\E[R_n]$.
\begin{lemma}\label{lemma:Rn-properties}
Let $R_n$ be defined as above and $p\in (0,1)$.
Then, for $n\to\infty$,
$$
\E[R_n]=\frac1{\log{p^{-1}}}\log{n}
		 +\frac{\log(1\!-\!p)}{\log{p^{-1}}}
		 +\frac\gamma{\log{p^{-1}}}
		 -\frac12+r_{p}(n)+o(1) \,,
$$
where $\gamma$ is the Euler-Mascheroni constant and
$r_{p}(n)$ is a periodic function which remains bounded and, more precisely, satisfies for all $n$
$$
|r_{p}(n)| \le  \frac1{2\pi}\sqrt{\theta} \frac{e^{-\theta}}{(1-e^{-\theta})^2}\, \qquad \text{with $\theta = \frac{\pi^2}{\log{p^{-1}}}$.}
$$
\end{lemma}

Recall that we denote by $A$ the set of active sensors and that the sensors are sorted according to their location.
The cost $C(\xeq)$ is tightly related with the lengths of runs of failures, and in particular the maximum run-length.
For a given set $A$ of active sensors, denote by $R$ the longest run-length of failures for that set $A$ (elements of $[n]$ not in $A$): $R$ is thus the realization of $R_n$ corresponding to $A$.
Notice that if $1 \in A$ and $n \in A$, then the coverage cost
is precisely determined by the longest run of failures, since
$ C_0(\xeq_A) = \frac{R+1}{2n} $.
However, when a failure occurs in sensors $1$ or $n$ (or both), the runs of failures involving border sensors contribute to the cost by a larger amount.
Denote by $L_{\mathrm{i}}$ and $L_{\mathrm{f}}$ the lengths of the runs of failures involving the initial sensor 1 and the final sensor $n$, respectively,
namely,  $L_{\mathrm{i}} = A_1 -1$ and $L_{\mathrm{f}} = n-A_{|A|}$ for $A \ne \emptyset$, and $L_{\mathrm{i}} = L_{\mathrm{f}} =  n$ for $A = \emptyset$.
Now notice that, for all $A \ne \emptyset$,
\[ C_0(\xeq_A) = \max \left\{ \frac{R+1}{2n} , \frac{2 L_{\mathrm{i}}+1}{2n} ,  \frac{2 L_{\mathrm{f}}+1}{2n} \right \} \,.\]
For the case where $A = \emptyset$, recall that $C_0(\xeq_{\emptyset}) = 1$.
Hence, for all $A$, we have the following bounds:
\[ C_0(\xeq_A) \ge  \frac{R+1}{2n} \]
and
\[ C_0(\xeq_A) \le \max \left\{ \frac{R+1}{2n} , \frac{2 L_{\mathrm{i}}+1}{2n} ,  \frac{2 L_{\mathrm{f}}+1}{2n} \right \} 
\le \frac{R+1}{2n} + \frac{2 L_{\mathrm{i}}+1}{2n} +  \frac{2 L_{\mathrm{f}}+1}{2n} \,.\]

The bounds on the averaged cost $C(\xeq)$ are then obtained by taking the expectation.
Notice that, with the failure model from Problem~\ref{prob:mainprob}
the maximum run-length $R$ is the above-described random variable $R_n$, and hence its average satisfies Lemma~\ref{lemma:Rn-properties}.
For the initial and final run-lengths, they are truncated geometric r.v.'s, in the following sense.
Let $X$ be a geometric r.v. of parameter $p$, namely $\Pr(X=k)=p^k (1-p)$.
Now notice that $\Pr(L_{\mathrm{i}} =k)$ and $\Pr(L_{\mathrm{f}} =k)$ are equal to $\Pr(X=k)$ for $k <n$, to $\Pr(X\ge n)$ for $k =n$ and to $0$ for larger $k$,
so that $\E L_{\mathrm{i}} = \E L_{\mathrm{f}} \le \E X = \frac{p}{1-p}$.

We can now conclude the proof: for the lower bound
\[ C(\xeq) \ge  \frac{\E R_n+1}{2n}  = 
\frac{1}{2n} \left(\frac{1}{\log{p^{-1}}} \log{n} +O(1)\right) \quad \text{for $n\to\infty$} \,,\]
while for the upper bound
\[ C(\xeq) \le \frac{\E R_n+1}{2n} + 2 \frac{2 \E X +1}{2n} = 
\frac{1}{2n} \left(\frac{1}{\log{p^{-1}}} \log{n} +O(1)\right) \qquad \text{for $n\to\infty$} \,.\]

\subsection{Coverage on a circle and proof of~\eqref{eq:almost-optimal-bound}}
\label{sec:line-circle}

In order to complete the proof of Theorem~\ref{th:equispace-logn} we
introduce a proxy model. Instead of covering the unit interval, this
time we attempt to find a good coverage on a circle with circumference~1.
If we represent the locations by values in $[0,1]$, this means that the distance between two points $x,y\in [0,1]$ is $\min(\abs{y-x}, 1-\abs{y-x})$.
Employing this distance to determine the cost as in~\eqref{eq:def_C0_xA} leads to define the following problem.
\begin{prob}[Independent failures -- Circle]\label{prob:circle}
For given $p\in (0,1)$ and $n\in \natural$, find $x\in [0,1]^n$ that minimizes
$
\tilde C(x) = \sum_{A \subseteq[n]} \Pr(E_A) \tilde C_0(x_A),
$
where $\Pr(E_A)=p^{n-\abs{A}}(1-p)^{\abs{A}}$ and 
\begin{equation}\label{eq:sorted_C0_circle}
\tilde C_0(x) = \max \left\{\frac{1}{2} (1 - x_n + x_1),  \max_{i=1, \dots, n-1} \frac{1}{2}(x_{i+1}-x_i).\right\}
\end{equation}
\end{prob}

Problem~\ref{prob:circle} can also be formulated as a linear problem; a result similar to Theorem~\ref{prop:linprog} with a minor modification to constraints~\eqref{constr:def_cost_boundary_in_thm_linprog} can be proved exactly in the same way.
\begin{corollary}[Linear program -- Circle]
Let $n\in \natural$ and $p\in (0,1)$. The (ordered) vector $x^*\in [0,1]^n$ is the optimal solution of Problem~\ref{prob:circle} if and only if there exists a vector $w^*\in \real^{{2^n}-1} $ such that $(x^*,w^*)$ is an optimal solution to the following linear program:
\begin{align}\nonumber
\min &\sum_{A\neq\emptyset} \Pr(E_A)w_A \\
s.t.\nonumber\\
 &0\leq x_1\leq x_2\leq \dots \leq x_n\leq 1,\nonumber\\
\text{and } &\forall A \subseteq [n],~A\neq \emptyset,\nonumber\\
&w_A \geq \frac{1}{2}(x_{A_{k+1}}- x_{A_k}), \hspace{.2cm} \text{for } k =1,\dots,|A|-1,\nonumber\\
&w_A \geq \frac{1}{2} (1- x_{A_{\abs{A}}} +x_{A_1}).\label{const_cycle_boundary}
\end{align}
\end{corollary}

We now show that the equispaced solution $\xeq$ is the optimal sensor placement for the circle. We will then
relate it to the original problem on the line.

\begin{proposition}[Optimal solution -- Circle]\label{prop:equidist_optimal_cycle}
The equispaced sensor placement $\xeq$ is the only optimal sensor
placement (up to translation) on the circle for Problem~\ref{prob:circle}.
\end{proposition}

\begin{proof}
  The linear program nature of the problem allows us to combine different sensor
  placements. Given $(x,w)$ and $(x',w')$ we may form
  $\left(\frac{x+x'}{2},\frac{w+w'}{2}\right)$. This is a valid point
  of the polytope of constraints, and the cost is between the cost of the two
  initial placements.
  On the other hand, using the symmetry of the circle it follows that
  the rotation of $x$ (formally a translation modulo 1) does not change the associated cost, even though $w$ may need to be changed appropriately. Without loss of generality, we assume thus $x_1 = 0$.

  Let us fix an initial $(x,w)$ and define the rotated versions
  $x^1,x^2,\ldots,x^n$ such that $x_k$ becomes $(x^k)_1=0$. For every
  $x^k$ we have the corresponding best $w^k$ which all give the same cost.
  We want a closer look on their average
  $$(x^*,w^*) = \left(\frac{1}{n}\sum_{k=1}^n x^k,
    \frac{1}{n}\sum_{k=1}^n w^k\right).$$
  Using our previous observations, this is a valid sensor placement
  and $w^*$ has the same cost as any $w^k$.
  But what is this $x^*$? Let us check the distance of two
  consecutive sensors (out of bound indices and distances have to be interpreted appropriately):
  $$x^*_{i+1}-x^*_i = \frac{1}{n}\sum_{k=1}^n (x^k_{i+1}-x^k_i) = \frac{1}{n}\sum_{k=1}^n (x_{k+i}-x_{k+i-1}) = \frac{1}{n}.$$
  Therefore $x^* = \xeq$. This already shows that the equispaced placement 
  is optimal. We now show that it is the only optimal solution.\\ For the sensor
  placement $x^*$ the accompanying $w^*$ is not necessarily the best possible. We
  claim that whenever $x^1$ is not equispaced, there is a
  $\tilde{w}^*$ such that $(x^*,\tilde{w}^*)$ is valid and strictly
  cheaper than $(x,w)$.
  When $x$ is not equispaced, it means that there are two
  consecutive sensors which are more than $1/n$ apart. In other words,
  $w_{[n]}>1/(2n)$. By the construction above, we get
  $w^*_{[n]}=w_{[n]}>1/(2n)$. On the other hand, we know
    that we can decrease $w^*_{[n]}$ to $1/2n$ for the equispaced
    placement without violating any constraints. Define 
  \begin{align*}
    \tilde{w}^*_{[n]} &= 1/(2n),\\
    \tilde{w}^*_{A} &= w^*_{A} \quad \text{otherwise.} 
  \end{align*}
  This way $(x^*,\tilde{w}^*)$ is a valid point of
  the polytope.
  The costs of the different settings compare as follows:
  \begin{align*}
  \tilde{C}(x) &= \Pr(E_\emptyset) + \sum_{A\neq\emptyset}\Pr(E_A)w_A = \Pr(E_\emptyset) + \sum_{A\neq\emptyset}\Pr(E_A)w^*_A \\
  &\ge \Pr(E_\emptyset) + \sum_{A\neq\emptyset}\Pr(E_A)\tilde{w}^*_A \ge
  \tilde{C}(x^*).
      \end{align*}
  This becomes a strict inequality whenever $\Pr([n])>0$. Consequently
  $\Pr([n])>0$ is a sufficient condition for $x^*=\xeq$ to be the
  strong optimum. This condition obviously holds for independent
  failures which concludes our proof.
\end{proof}

\begin{rem}\label{rem:xeq-optimal}
The same proof shows that $\xeq$ is an optimal sensor placement for
any variation of Problem~\ref{prob:circle} where $\Pr(E_A)$ is independent of the positions of the sensors and invariant under permutation.
Moreover, if there is a nonzero probability that all sensors are active, it is the only optimal placement, up to translations.
\end{rem}

Next, we show that the optimal cost of our initial Problem~\ref{prob:mainprob} lies between the cost $\tilde C(\xeq)$ of the (optimal) equispaced solution $\xeq$ for Problem~\ref{prob:circle} on the circle and the cost $C(\xeq)$ of the same distribution for Problem~\ref{prob:mainprob}. For this purpose, we need the following lemma providing a bound on the difference of cost for each set $A$ of active sensors, which will also prove useful later. 

\begin{lemma}\label{lem:compare_circle_line_one_set}
Let $A$ be a non-empty set of active sensors and $x_A$ their positions. There holds
$
\tilde C_0(x_A) \leq C_0(x_A),
$
Moreover, if $\tilde C_0(x_A) < C_0(x_A)$, then $C_0(x_A) = \max \{ x_{A_1}, 1-x_{A_{\abs{A}}}\}$.
\end{lemma}
\begin{proof}
Consider~\eqref{eq:sorted_C0_A}.
Adding the average of the first two terms in the set on which the maximum is taken does not affect the value of the maximum. We have therefore 
\begin{equation}\label{eq:C_0_extended}
C_0(x_A)=  \max \left\{ x_{A_1}, (1-x_{A_{\abs{A}}}), \frac{1}{2}(1 + x_{A_1} - x_{A_{\abs{A}}}),  \max_{k=1, \dots, \abs{A} -1}  \frac{1}{2}(x_{A_{k+1}}-x_{A_k})\right\}.\end{equation}
Observe that every quantity appearing in the definition 
\begin{equation}\label{eq:tilde_CO_extended}
\tilde C_0(x_A)=  \max \left\{ \frac{1}{2}(1 + x_{A_1} - x_{A_{\abs{A}}}),  \max_{k=1, \dots, \abs{A} -1}\frac{1}{2} (x_{A_{k+1}}-x_{A_k})\right\}
\end{equation}
also appears in~\eqref{eq:C_0_extended}. Therefore, we have $\tilde C_0(x) \leq C_0(x_A)$. Moreover, in case this inequality is strict, $C_0(x_A)$ must be equal to one of the elements that appear in~\eqref{eq:C_0_extended} but not 
in~\eqref{eq:tilde_CO_extended}, that is, either $x_{A_1}$ or $1-x_{A_{\abs{A}}}$.
\end{proof}

\begin{lemma}\label{lem:3bounds}
Let $x^*$ be an optimal solution to Problem~\ref{prob:mainprob} for given $n$ and $p$. There holds 
$$\tilde C (\xeq) \leq \tilde C (x^*) \leq C(x^*) \leq C(\xeq).$$
\end{lemma}
\begin{proof}
The inequality $C(x^*)\leq C(\xeq)$ follows from the optimality of $x^*$ for Problem~\ref{prob:mainprob}. Similarly, $\tilde C(\xeq) \leq \tilde C(x^*)$ follows from the optimality of $\xeq$ for Problem~\ref{prob:circle} proved in Theorem~\ref{prop:equidist_optimal_cycle}. Finally, since $C(x) - \tilde C(x) = \sum_{A\neq\emptyset}\Pr(E_A)(C_0(x_A)-\tilde C_0(x_A))$, it follows from Lemma~\ref{lem:compare_circle_line_one_set} that $\tilde C(x) \leq C(x)$ for every $x$, and in particular that $C (x^*) \leq C(x^*)$.
\end{proof}

Thanks to Lemma~\ref{lem:3bounds}, now we just have to evaluate the difference between the cost of the equispaced solution $\xeq$ in Problems~\ref{prob:mainprob} and~\ref{prob:circle}. 

\begin{lemma}\label{lem:bound_eq_line-eq_circle}
For any $n\in \natural$, $p\in (0,1)$, there holds
$$
C(\xeq) \leq \tilde C(\xeq)  +\frac{2}{n}\frac{p}{1-p}.
$$
\end{lemma}
\begin{proof}
We first consider a (non-empty) set of active sensors $A$ and find a bound on the difference of cost $C_0(\xeq_A) - \tilde C_0(\xeq_A)$. Observe first that $\xeq_i = \frac{1}{2n} (2i - 1)$, and therefore that $\tilde C_0(\xeq_A) \geq \frac{1}{2n}$ in all cases. Suppose now that $C_0(\xeq_A)$ and $\tilde C_0(\xeq_A)$ are different. It follows in that case from Lemma~\ref{lem:compare_circle_line_one_set} that $C(\xeq_A) >  \tilde C(\xeq_A)$, and that 
$$C_0(\xeq_A) = \max (x_{A_1},1 - x_{A_{\abs{A}}})=  \frac{1}{2n }\max (2A_1 -1 ,1 + 2(n-{A_{\abs{A}}})).$$
Whenever $C_0(\xeq_A) \neq \tilde C_0(\xeq_A)$, we have thus \begin{equation}\label{eq:bound_one_event}
C_0(\xeq_A) - \tilde C_0(\xeq_A) \leq \frac{1}{2n }\max (2A_1 -2 , 2(n-{A_{\abs{A}}}) \leq   \frac{1}{n }(A_1 -1 + n - A_{\abs{A}}).
\end{equation}
When $C_0(\xeq_A) = \tilde C_0(\xeq_A)$, the inequality also holds since the right-hand side of~\eqref{eq:bound_one_event} is non-negative.

We now sum the inequality~\eqref{eq:bound_one_event} over all events and use the symmetry of our problem to obtain 
\begin{align}
C(\xeq) - \tilde C(\xeq) &\leq \frac{1}{n}\sum_{A\neq\emptyset}\Pr(E_A) (A_1 -1) + \frac{1}{n } \sum_{A\neq\emptyset}\Pr(E_A) (n - A_{\abs{A}})\nonumber\\& = \frac{2}{n}\sum_{A\neq\emptyset}\Pr(E_A) (A_1 -1)\nonumber \\
&= \frac{2}{n} \sum_{k=1}^n \Pr(A_1 =k)  (k-1), \label{eq:bound_difference_togeometric} 
\end{align}
where the event $A_1= k$ implicitly implies that $A$ is non-empty. 
Observe  that the probability for the $k^{th}$ sensor to be the first active one is $p^{k-1}(1-p)$. Therefore, the expression $\sum_{k=1}^n \Pr(A_1 =k)  (k-1)$ is the expected value of a truncated geometric random variable (i.e. a geometric random variable whose value is set to 0 if it exceeds $n$), and is bounded by $\frac{p}{1-p}$. Reintroducing this into~\eqref{eq:bound_difference_togeometric} leads to the desired result $C(\xeq) - \tilde C(\xeq) \leq \frac{2}{n} \frac{p}{1-p}$.
\end{proof}

The inequality~\eqref{eq:almost-optimal-bound} in Theorem~\ref{th:equispace-logn} 
follows then from the combination of Lemmas~\ref{lem:3bounds} and~\ref{lem:bound_eq_line-eq_circle}.

\section{Extreme values of $p$}\label{sect:extreme_p}

In this section, we study the optimal placement when $p$ takes on extreme values, either close to 0 or to 1. Our first result gives the optimal placements under such conditions.

\newcommand{\xsingle}{\supscr{x}{sgl}}
\begin{proposition}[Small and large $p$]\label{prop:neighborhoods}
If $p$ is in a neighborhood of 0, then the equispaced placement $\xeq$ is optimal. Similarly, if $p$ is in a neighborhood of 1, then the optimal placement is $\xsingle$, where $\xsingle_i=\frac12$ for all $i\in[n]$.
  \label{prp:pnearextreme}
\end{proposition}
\begin{proof}
We rely on the linear program formulation in \S~\ref{sect:linprog}. We have seen that the polytope of constraints is independent of $p$, and that the cost vector evolves continuously with $p$.
For $p=0$ we know that the unique optimal solution is $\xeq$. This means that for any other vertex $x$ of the polytope of constraints we have
$$C(\xeq) < C(x).$$
Let us denote the set of vertices of the polytope of constraints by $V$. Knowing that $V$ is finite, we get
$$C(\xeq) < \min_{x\in V\setminus \{\xeq\}}C(x).$$
The strict inequality and the continuity of the cost function with respect to $p$ imply that, for a sufficiently small perturbation of the cost vector,
 $\xeq$ will remain the optimal placement. In other words, $\xeq$ is optimal as long as $p$ is in a sufficiently small neighborhood of 0.

For large failure probability $p=1-\eps$ 
the most relevant events are those with just one active sensor,
in the sense that any $A$ with size two or more has $\Pr(E_A) = O\left(\eps^2\right)$.
Then,
$$C(x) = (1-\eps)^n + \eps (1-\eps)^{n-1} \sum_{i=1}^n\max(x_i,1-x_i)+O\left(\eps^2\right).$$
This holds for any placement $x$, so in particular for all $x \in V$. Clearly $\xsingle$ is strictly optimal
concerning the main term $\sum_{i=1}^n\max(x_i,1-x_i)$. Recalling that $V$ is finite, this implies that one can find a sufficiently small $\bar \eps$ such that,
for all $\eps \le  \bar \eps$, $C(\xsingle) < C(x)$ for all $x \in V \setminus \{\xsingle\}$, i.e., $\xsingle$ is optimal.
\end{proof}

The next two results provide estimates on the sizes of the neighborhoods in Proposition~\ref{prop:neighborhoods}, showing that their sizes asymptotically vanish as $n$ diverges. 
Their proofs, presented in Appendices \ref{ap:propequidistsmall} and \ref{ap:propequidistlarge}, rely on comparing $\xeq$ and $\xsingle$ with alternative placements, and showing that the former are not optimal when $p$ differs from respectively 0 or 1 by more than a certain value that decays in $O(1/n)$.

\begin{proposition}[Neighborhood of 0]
The neighborhood of 0 where $\xeq$ is optimal is at most $c_0/n$ long,
with some constant $c_0>0$.
\label{prop:equidistsmallp}  
\end{proposition}

\medskip

\begin{proposition}[Neighborhood of 1]
The neighborhood of 1 where the single cluster placement $\xsingle$ is optimal is at most $3/n$ long.
\label{prop:equidistlargep}  
\end{proposition}

The proof of Proposition~\ref{prop:equidistlargep} actually shows the slightly stronger result that $\xsingle$ is sub-optimal for any $p<1-3/n$. In other words, it does not become optimal again for smaller values of $p$.

\section{Performance of a random placement}\label{sect:random}
In this section we consider a random placement $\xrand$ of the sensors. More
precisely, the positions $\xrand_1$, $\dots$, $\xrand_n$ are i.i.d.~random
variables, uniformly distributed in the interval $[0,1]$.
Notice that (differently from $x$ in the rest of the paper), here $\xrand$ has entries which are not ordered, so that
the cost definition in~\eqref{eq:def_C0_xA} applies, while the one in~\eqref{eq:sorted_C0_A} does not.

The following result describes the asymptotic behavior of $\E [C(\xrand)]$, where $\E$ denotes expectation with respect to the random positions of sensors. Note that the cost $C(x)$ defined in~\eqref{eq:def-average-cost} is itself averaged with respect to sensor failures.

\begin{theorem}[Cost of random placement]\label{thm:random}
Let $\xrand$ be the above-defined random sensor placement. Then,
\[
 \E [C(\xrand)] = 
\frac{1}{2(1-p)} \frac{\log n}{n}  + O\!\left(\frac{1}{n}\right) \qquad \text{for $n\to\infty$}
\,.
\]
\end{theorem}

\medskip

From Theorem~\ref{thm:random} we can argue that $\E[C(\xrand)]$ has the same order of growth as $C(\xeq)$, but with a larger constant, thus leading to an asymptotically worse performance: this comparison is illustrated in Figure~\ref{fig:rand-vs-equi}. 
The rest of the section describes the main steps of the proof, while some lengthier details are postponed to the appendix.

\begin{figure}
\centering
\psfrag{uniform}{$C(\xeq)$}
\psfrag{random}{$\E[ C(\xrand)]$}
\psfrag{principal part}{$\frac1{2\log{p^{-1}}}\frac{\log{n}}{n}$}
\psfrag{principal part 2}{$\frac1{2 (1-p)}\frac{\log{n}}{n}$}

\psfrag{n}{$n$}
\psfrag{cost}{}
\includegraphics[width=.6\columnwidth]{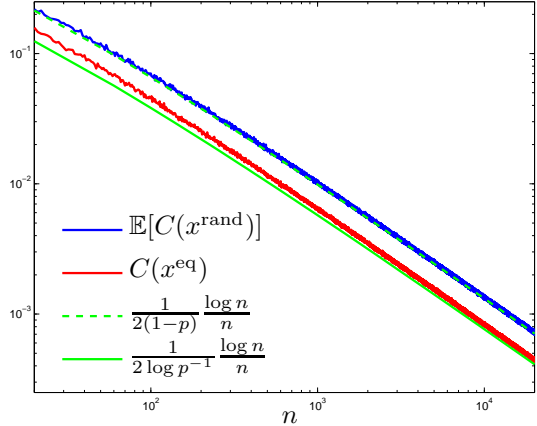}
\caption{The plot compares $\E [C(\xrand)]$, $C(\xeq)$, and their approximations according to Theorems~\ref{thm:random} and~\ref{th:equispace-logn}, respectively. The expected costs are simulated as Monte Carlo averages over 100 independent realizations of the placements and of the failures, taking $p=0.3$.}
\label{fig:rand-vs-equi}
\end{figure}

From the definition in~\eqref{eq:def-average-cost} and by linearity of expectation,
\[ \E[ C(\xrand)] = \sum_{A \subseteq [n]} \Pr(E_A) \E [C_0(\xrand_A)] \,. \]
The key remark is that all sets $A$ having the same cardinality $m$ have the same average cost $\E [C_0(\xrand_A)]$,
which corresponds to the average cost of a vector of $m$ active sensors in random positions.
Then, we define $\xrandm$ as a vector with $m$ entries $\xrandm_1, \dots, \xrandm_m$, i.i.d.\ uniform in $[0,1]$.
With this notation, $\E [C_0(\xrand_A)] = \E [C_0(\xrandm)]$ with $m = |A|$, so that
\begin{equation} \label{eq:cost-xrandm}
\E[ C(\xrand)] = \sum_{m=0}^n \Pr(|A|=m)  \E [C_0(\xrandm)] \,.
\end{equation}

Hence, we focus on finding bounds for $\E [C_0(\xrandm)]$. 
To do so, we make use of classic results about lengths of segments when cutting a rope at random points,
as described below.
We introduce the notation $V_1, \dots, V_{m+1}$ for the lengths of the segments obtained when cutting the $[0,1]$ interval at points from $\xrandm$.
More precisely, let $y = (y_1, \dots, y_m)$ be the vector obtained re-ordering entries of $\xrandm$ in non-decreasing order; also define $y_0 =0$ and $y_{m+1} =1$; finally define $V_i = y_i - y_{i-1}$, for $i=1, \dots, m+1$.
The average cost  $\E [C_0(\xrandm)]$ is related to the distribution of the segment lengths $V_1, \dots, V_{m+1}$, as follows.
\begin{lemma}  \label{lemma:EC0randm-rope}
For any $m \ge 1$,
\begin{equation} \label{eq:lemma-random-b}
 \E [C_0(\xrandm)] = \int_0^1  \Pr (C_0(\xrandm)>v) \mathrm{d}v \, ,
\end{equation}
where
$ \displaystyle 
\Pr (C_0(\xrandm)>v) = 
\Pr \left( \{V_1 > v \} \cup  \left\{ \frac{V_2}{2} > v \right\} \cup \dots \cup \left\{\frac{V_{m}}{2} > v \right\} \cup \{V_{m+1} > v \}  \right) 
$. \vspace{1mm}
\end{lemma}
\begin{proof}
By computing the expectation as the integral of the survival function, \eqref{eq:lemma-random-b} immediately follows.
From~\eqref{eq:sorted_C0_A} applied to $y$, we have
$ 
C_0(\xrandm) = \max\left(V_1, \frac{V_2}{2}, \frac{V_3}{2}, \dots, \frac{V_{m}}{2}, V_{m+1}\right) 
$,
which implies the second equality.
\end{proof}

We will then take advantage of the following result about the distribution of the segment lengths $V_1, \dots, V_{m+1}$.
\begin{lemma}[{\cite[Sect.~6.4]{book_OrderStatistics}}] \label{lem:cuttingrope}
Let $V_1, \dots, V_{m+1}$ be the above-defined segment lengths.
Given $r\le m+1$ non-negative parameters  $c_1, \dots, c_r$ such that $\sum_i c_i \le 1$, and distinct indices $i_1,\dots, i_r \in [m+1]$, then
\[ \Pr(V_{i_1}>c_1, \dots, V_{i_r}>c_r)  = (1-c_1-\dots-c_r)^{m} \,.\]
\end{lemma}

The above lemmas, together with inclusion-exclusion principle, allow us to find the following bounds for $\E[ C_0(\xrandm)]$.
The bounds involve the harmonic numbers $H_m = \sum_{h=1}^m \frac{1}{h}$. 
The details of the proof are given in Appendix \ref{ap:boundsEC0}.
\begin{lemma} \label{lemma:bounds-EC0xrandm}
For all $m\ge 0$,
\[  \E[ C_0(\xrandm)] \ge
\frac{H_{m+1}}{2(m+1)} \,. \]
Moreover, for all $m \ge 2$,
 \[ \E[ C_0(\xrandm)] \le
\frac{H_{m-1} +4}{2(m+1)} \,. \]
\end{lemma}

Then, using Lemma~\ref{lemma:bounds-EC0xrandm}, we can find the following bounds for $\E (C(\xrand))$.
The proof is described in detail in Appendix \ref{ap:boundsCXrand}.
\begin{lemma}  \label{lemma:bounds-Cxrand}
For any $\eps \in (0,p)$,
\[ 
\E [C(\xrand)] \ge 
(1 - e^{-2 \eps^2 n}) \frac{H_{\lceil (1-p+\eps)n \rceil}}{2 \lceil(1-p+\eps)n \rceil}
\,. \]
Moreover, for any $\eps \in (0,1-p)$,
\[ 
\E [C(\xrand)] \le 
e^{-2 \eps^2 n} + 
	\frac{H_{\lceil (1-p-\eps)n \rceil} +4}{2 \lceil (1-p-\eps)n \rceil +2}  
\,. \] 
\end{lemma}

The statement of Theorem~\ref{thm:random}  follows from Lemma~\ref{lemma:bounds-Cxrand}, by taking $\eps = \sqrt{ \frac{\log n}{n}}$ and by exploiting the fact that the asymptotic growth of harmonic numbers is $H_m \sim \log m$ for $m \to \infty$.

\section{Cort\'es model}\label{sect:cortes}

As mentioned in the Introduction, the paper~\cite{JC:12} studies the
coverage problem in one-dimension with the following failure model for the sensors. The number of failing sensors is known (and indicated with $k$) but {\em which} sensors fail is unknown and random: more precisely, the set of the $k$ failing sensors is sampled from a uniform distribution over the subsets of $\until{n}$ with $k$ elements. The problem can be summarized as follows.

\begin{prob}[Constant number of failures -- Cort\'es model]\label{prob:Cortes}
For given positive integers $k, n$ with $k<n$, find $x\in [0,1]^n$ that minimizes
$C(x) = \Pr(E_{\emptyset})+\sum_{A\neq\emptyset} \Pr(E_A) C_0(x_A),$
where  $C_0(x)$ is defined by~\eqref{eq:def_C0_xA} and $\Pr(E_A)= \prt{\frac{n!}{k!(n-k)!}}^{\! -1}$ if $\abs{A} = k$ and 0 otherwise.
\end{prob}

Observe that the only difference with our Problem~\ref{prob:mainprob}
is that the probabilities $\Pr(E_A)$ have changed. The following lemma
indicates how the Cort\'es model can be approximated by the
independent failure model.

\begin{lemma}\label{lem:bound_cortes_indep}
For any $x\in [0,1]^n$, $k<n$ and $0 < \epsilon < \min (k/n, 1-k/n)$,
$$
C^{\frac{k}{n}-\epsilon}(x) - e^{-2\epsilon n ^2}\leq C^{k,n}(x) \leq C^{\frac{k}{n}+\epsilon}(x)+e^{-2n \epsilon^2},
$$
were we use the notation $C^{k,n}$ for the cost of Cort\'es model and $C^p$ for the independent failure model. 
\end{lemma}
\begin{proof}
Problem~\ref{prob:Cortes} involves uniformly randomly selecting a subset $A$ of $k$ failed sensors among $n$ possible ones. One way of doing this is to first build a set $B$ obtained by selecting independently every sensor with a probability $p$. Then, if $\abs{B}>k$, one obtains $A$ by removing $\abs{B}-k$ uniformly randomly selected sensors from $B$. If on the other hand 
$\abs{B}<k$, one adds $k-\abs{B}$ randomly selected sensors to $B$. Observe that $A$ then always contains $k$ sensors, and that all sets $A$ with cardinality $k$ are equiprobable, so it is a valid selection process with respect to Problem~\ref{prob:Cortes}. 
The cost of $x$ can be decomposed as the contributions of the event $\abs{B} >k$ and $\abs{B} \leq k$.
\begin{equation}\label{eq:decompsition_cost_cortes}
C^{k,n}(x) = \E( C_0(x_A)) = \E\prt{ C_0(x_A)| \abs{B}\leq k} \Pr(\abs{B}\leq k)  + \E\prt{ C_0(x_A)| \abs{B}> k} \Pr(\abs{B}> k).
\end{equation}
When $\abs{B} \leq k$, the set $A$ contains the set $B$ from which it was built, and the cost $C_0(x_A)$ is thus smaller than or equal to $C_0(x_B)$. As a result, $\E\prt{ C_0(x_A)| \abs{B}\leq k} \leq \E\prt{ C_0(x_B)| \abs{B}\leq k}$. On the other hand, the cost $C_0(x_B)$ is always bounded by 1, and thus there always holds  $C_0(x_B) \leq 1 + C_0(x_A)$. In particular, $\E\prt{ C_0(x_A)| \abs{B}> k}\leq 1 + \E\prt{ C_0(x_A)| \abs{B}> k}$. Reintroducing these two bounds in~\eqref{eq:decompsition_cost_cortes} yields
\begin{align}
C^{k,n}(x) &\leq \E\prt{ C_0(x_B)| \abs{B}\leq k} \Pr(\abs{B}\leq k)  + \E\prt{ C_0(x_B)| \abs{B}> k} \Pr(\abs{B}> k) + \Pr(\abs{B}> k)\nonumber\\& = \E\prt{C_0(x_B)} + \Pr(\abs{B}> k) = C^p(x) +  \Pr(\abs{B}> k),\label{eq:almost_final_upper_bound_cortes}
\end{align}
where the last inequality follows from the fact that the sets $B$ are
built by randomly taking each sensor with a probability $p$ as in
Problem~\ref{prob:mainprob}.
Now the size of $B$ follows a binomial distribution with parameters $n$ and $p$. Hoeffding's inequality implies then that $\Pr(\abs{B}> k) \leq \exp(-2 \frac{(np-k)^2}{n})$. Taking $p=\frac{k}{n}-\epsilon$, we obtain $\Pr(\abs{B}> k) \leq e^{-2\epsilon^2n}$, and the upper bound of this lemma follows then from~\eqref{eq:almost_final_upper_bound_cortes}. The lower bound is obtained in a parallel way.
\end{proof}

Results analogous to those presented in the previous sections can then be obtained for the model in~\cite{JC:12}. We collect them in the following Theorem.
\begin{theorem}[Constant number of failures -- Cort\'es model]$ $
\begin{itemize}
\item[(a)] {\bf Linear program.} 
Theorem~\ref{prop:linprog} is directly valid for Problem~\ref{prob:Cortes}.
\item[(b)] {\bf Asymptotic cost of $\xeq$.}  For fixed $k/n$ and $n\to\infty$,
$C(\xeq)$ approximates $\frac{1}{2\log \frac{n}{k}}\frac {\log n}{n}.$
More precisely, for any $0<\epsilon<\min(k/n,
1-k/n)$ we have
$$\frac{1}{2\log \frac{n}{k-n\epsilon}}\frac {\log n}{n} +
O\left(\frac{1}{n}\right) \le C(\xeq) \le \frac{1}{2\log
  \frac{n}{k+n\epsilon}}\frac {\log n}{n} + O\left(\frac{1}{n}\right)
\quad \text{for $n\rightarrow\infty$},$$
where the $O(1/n)$ term can depend on $\epsilon$ and $k/n$.
\item[(c)] {\bf Near-optimality of $\xeq$.} 
Let $x^*$ be the optimal solution to Problem~\ref{prob:Cortes}. There holds $$ C(\xeq) \leq C(x^*) + \frac{2}{n}\frac{k}{n-k}.$$
\item[(d)] {\bf Asymptotic cost of $\xrand$.} The average cost of the random placement has the asymptotic behavior
\[ \E[C(\xrand)] = \frac{\log m}{2 m} + O\left( \frac{1}{m} \right)  \; \text{for $m \to \infty$} \,,\]
where $m=n-k$ is the number of active sensors.
\end{itemize}
\end{theorem}
\begin{proof}
(a) The proof of Theorem~\ref{prop:linprog} does not depend on the values of the probabilities $\Pr(E_A)$. It applies thus directly to other models of probabilities, including that of Problem~\ref{prob:Cortes}. Moreover, the polytope of admissible solutions does not depend on $\Pr(E_A)$ either. Therefore, whenever the optimal solution is unique, it must belong to the (finite) set of vertices of that polytope, independently of the model.

(b) This part of the result is obtained by combining the bound \eqref{eq:exp-equally-theorem_simple} in Theorem~\ref{th:equispace-logn} with Lemma~\ref{lem:bound_cortes_indep}.

(c) The proof follows the reasoning held in
\S~\ref{sec:perf_equidist}. Specifically we can introduce a
variation of Problem~\ref{prob:Cortes} on the
circle. As explained in Remark \ref{rem:xeq-optimal}, Proposition~\ref{prop:equidist_optimal_cycle} implies then that
$\xeq$ is an optimal solution of that problem (though not necessarily
the only one since the probability for all sensors to be active is
zero if $k>0$). Lemmas~\ref{lem:compare_circle_line_one_set}
and~\ref{lem:3bounds} can then directly be extended with the same
proof, so that $\tilde C(\xeq) \leq C(x^{*})\leq C(\xeq)$. The bound
(c) follows then from a variation of
Lemma~\ref{lem:bound_eq_line-eq_circle} showing that  $C(\xeq) -
\tilde C(\xeq)\leq\frac{2}{n}\frac{k}{n-k}$. 

(d) Similarly to the proof of Theorem~\ref{thm:random}, we get
$ \E C_0(\xrand_A) = \E  C_0(\xrandm) $ for any $A$ with $|A|=m$,
and hence also $\E C(\xrand) = \E C_0(\xrandm)$.
Then, applying Lemma~\ref{lemma:bounds-EC0xrandm}, we get
\[  \frac{H_{m+1}}{2(m+1)} 
\le \E C(\xrand) \le 
\frac{H_{m-1} +4}{2(m+1)} \,,\]
which concludes the proof.
\end{proof}

Our results on the asymptotic behavior of solutions to Problem~\ref{prob:Cortes} complement those in~\cite{JC:12}, which focus on general properties of the optimization problem and on deriving explicit formulas for certain values of $k,n$.

\section{Conclusion}\label{sect:conclusion}

In this paper we have presented our findings on a new
model of coverage by unreliable sensors, which extends the well-known
disk-coverage problem to allow for independent sensor failures. We have shown that the resulting optimization problem is a linear program,
thus solvable by standard methods. However, since the space of possible solutions grows exponentially with the size of the problem,
we do not know whether a solution can be found in a polynomial time.
Although the optimal solution can possibly be hard to find, and even if its properties are difficult to describe precisely, we have been
able to present a suboptimal solution which asymptotically achieves the optimal performance as the number of sensors grows to infinity.
Remarkably, this near-optimal solution is just the equispaced placement, which is optimal in the case without failures.
We have also compared the performance of random sensor placement to
the equispaced setting to find that there is a constant factor
deterioration of the cost: nevertheless, the rate of growth is the same as the
number of sensors increases.

This paper opens several research directions. The first natural direction is the extension to higher dimensions. As mentioned in the introduction, the coverage performance of two-dimensional sensor networks has been extensively investigated. Consistently with the results in~\cite{SK-THL-JB:04} and some preliminary results that we have obtained, we believe that both $C(\xeq)$ and $C(\xrand)$ are asymptotically proportional to $\sqrt{\frac{\log n}{n}}$. However, characterising the proportionality constant and its optimality is an open question. In fact, our optimality analysis hinges on the assumption of dimension one: crucially, the linear programming characterisation is unlikely to effectively extend to higher dimensions.
Secondly, in this paper we have chosen a min-max disk-coverage cost: different cost functions would lead to interesting alternative problems. For instance, one can consider the weighted integral of a non-decreasing function of the distance to the closest sensor. 
Thirdly, one might consider the case of heterogeneous sensors, where the failure probability can depend on the sensor itself or on its location.
Finally, a challenging question is finding feedback control laws that enable autonomous deployment of self-propelled sensors, in such a way to take random failure into account. This problem has been recently studied in relation to Cort\'es model in~\cite{MO-JC:13}, but is completely open for the failure model proposed in this paper.

\appendix

\section{Proof of Proposition~\ref{prop:equidistsmallp} }
\label{ap:propequidistsmall}
We propose the alternative sensor placement
\newcommand{\xalt}{\supscr{x}{alt}}
$$\xalt = \frac{1}{2n-2}(1,2,4,6,\ldots, 2n-2, 2n-4, 2n-3)$$ and we show that, for some $p=c/n$ and for sufficiently large $n$,
this placement gives a better (expected) cost than $\xeq$.
In order to do so, we estimate the cost difference $C(\xalt) - C(\xeq)$.

We first compare the cost difference for any fixed set of active sensors $A$.
If $A$ is empty, the two costs are trivially the same.
Now consider a non-empty fixed $A$,
and let $k$ be the length of the longest sequence of consecutive failed sensors among the \emph{middle} ones $2,3,\ldots,n-1$.
In the following, we are going to prove the following bounds:
\begin{itemize}
\item[(a)] $\displaystyle C_0(\xalt_A) - C_0(\xeq_A) \leq I_k:=\frac{k+1}{2n(n-1)}$;
\item[(b)] if $k\le 1 $ and sensor 1 fails, then $\displaystyle C_0(\xalt_A) - C_0(\xeq_A)\leq -J := -\frac{n-5}{2n(n-1)}$.
\end{itemize}
Both bounds are based on the following observation:
\begin{equation} \label{eq:xalt}
C_0(\xeq_A) = \max\left\{ \frac{k+1}{2n}, \frac{2h+1}{2n} \right\}\,,
\quad
C_0(\xalt_A) \le \max\left\{ \frac{k+1}{2(n-1)}, \frac{h}{n-1} \right\}\,,
\end{equation}
where $h = \max\{ A_1-1, n-A_{|A|}\}$ is the longest between the runs of failures involving the first and last sensor. 
Notice that $0 \le h \le k+1$.

To prove (a), consider two cases. If $k \ge 2h$, then both maxima in \eqref{eq:xalt} are achieved by the term involving $k$, and
$C_0(\xalt_A) - C_0(\xeq_A) \le \frac{k+1}{2(n-1)} - \frac{k+1}{2n}= I_k$. If $k < 2h$, then 
both maxima in \eqref{eq:xalt} are achieved by the term involving $h$, and
$C_0(\xalt_A) - C_0(\xeq_A) \le \frac{h}{n-1} - \frac{2h+1}{2n} = \frac{2h - (n-1)}{2n(n-1)}$; the claim then follows using 
in the numerator the bounds $h \le k+1$ and $n-1 \ge k+1$ (the latter is true since by definition $k \le n-2$).

To prove (b), notice that the assumption that sensor 1 fails implies $h \ge 1$; 
also recall that $h \le k+1$, and that by assumption $k \le 1$, so that we have
$k \le 1 \le h \le 2$. In this case, both maxima in \eqref{eq:xalt} are achieved by the term involving $h$, and
$C_0(\xalt_A) - C_0(\xeq_A) \le  \frac{2h - (n-1)}{2n(n-1)}$; the claim follows from the bound $h \le 2$.

Now we come back to the averaged costs.
We denote by $E_k$ the set of sets $A$ for which the longest sequence of failed sensors among the middle ones has length $k$, 
and by $F_1$ the set of sets $A$ for which the sensor $1$ fails. 
We study 
\[ C(\xalt_A) - C(\xeq_A) 
= \sum_{k=0}^{n-2} \Pr(E_k) \E [ C_0(\xalt_A) - C_0(\xeq_A)  | E_k ]  \,. \]
For all terms with $k \ge 2$, we use the bound (a) to get $\E [ C_0(\xalt_A) - C_0(\xeq_A)  | E_k ] \le I_k$.
For $k=0$ and $k=1$, we separate the case where sensor 1 fails or is active:
\begin{align*}
&\sum_{k=0}^{1} \Pr(E_k) \E [ C_0(\xalt_A) - C_0(\xeq_A)  | E_k ] \\
=&\sum_{k=0}^{1} \Pr(E_k\cap F_1) \E [ C_0(\xalt_A) - C_0(\xeq_A)  | E_k\cap F_1 ] +
\sum_{k=0}^{1} \Pr(E_k\cap \bar F_1) \E [ C_0(\xalt_A) - C_0(\xeq_A)  | E_k\cap \bar F_1 ]  .
\end{align*}
For the first term, we can use the tighter bound (b), to get 
$ \E [ C_0(\xalt_A) - C_0(\xeq_A)  | E_k\cap F_1 ] \le -J$; for the second term we use bound (a),
together with the remark that $I_0 < I_1$, to get $\E [ C_0(\xalt_A) - C_0(\xeq_A)  | E_k\cap \bar F_1 ] \le I_1$.
Notice that $E_0$ and $E_1$ are disjoint, and that $F_1$ is and independent event from any $E_k$ since sensor failures are independent by assumption,
with $\Pr(F_1) = p$. Hence, we have 
$\sum_{k=0}^{1} \Pr(E_k\cap F_1) \E [ C_0(\xalt_A) - C_0(\xeq_A)  | E_k\cap F_1 ] 
\le 
- \Pr(E_0 \cup E_1) p J$ and
$\sum_{k=0}^{1} \Pr(E_k\cap \bar F_1) \E [ C_0(\xalt_A) - C_0(\xeq_A)  | E_k\cap \bar F_1 ]
\le 
\Pr(E_0 \cup E_1) (1-p) I_1$.

Collecting all terms, we have
\[
C(\xalt) - C(\xeq)\leq 
\sum_{k=2}^{n-2} \Pr(E_k) I_k + \Pr(E_0 \cup E_1)(1-p)I_1 - \Pr(E_0 \cup E_1)p \,,
\]
which we can re-write as
$$
C(\xalt) - C(\xeq)\leq \underbrace{\sum_{k=2}^{n-2} \Pr(E_k) I_k - \Pr(E_0 \cup E_1)  p \frac{J}{2}}_{(\alpha)} \\
 + \Pr(E_0 \cup E_1) \underbrace{((1-p)I_1 - p\frac{J}{2})}_{(\beta)}.
$$
We now show that $(\alpha),(\beta)$ are both negative when $p=\frac{c}{n}$ for a suitable $c$ and sufficiently large $n$.
Substituting the values of $I_1$ and $J$ in $(\beta)$ leads to 
$
(\beta) =\frac{1}{4n(n-1)}\prt{4-pn+p},
$
which is negative for sufficiently large $n$ when $p= c/n$ for any $c>4$. 

To analyze $(\alpha)$, we start by bounding $\Pr(E_k)$. There are $n-k-1$ possible sequences of~$k$ consecutive middle sensors and the probability that all the sensors fail in one such sequence is $p^k$. Therefore, 
\begin{equation}\label{eq:bound_Ek}
\Pr(E_k) \leq (n-k-1)p^k<n p^k
\end{equation}
and as a consequence
\begin{equation}\label{eq:bound_E0E1}
\Pr(E_0\cup E_1) = 1 - \sum_{k=2}^{n-2}\Pr(E_k) > 1 - \sum_{k=2}^\infty n p^k=1 - n\frac{p^2}{1-p}.
\end{equation}
The first part of inequality \eqref{eq:bound_Ek} allows bounding the first term in $(\alpha)$:
\begin{align*}
\sum_{k=2}^{n-2}\Pr(E_k)I_k 
&< \sum_{k=2}^{n-2} (n-k-1)p^k \frac{k+1}{2n(n-1)} 
< \frac{1}{2n}\sum_{k=2}^{n-2} p^k (k+1)\\
 &< \frac{1}{2n}\left( \frac{1}{(1-p)^2} - 2p- 1\right) 
 = \frac{3p^2-2p^3}{2n(1-p)^2}.
\end{align*}
Re-introducing this bound in $(\alpha)$ and using \eqref{eq:bound_E0E1}  leads then to 
\begin{align*}
(\alpha) &<  \frac{3p^2-2p^3}{2n(1-p)^2} - p\left(1 -
  n\frac{p^2}{1-p}\right)\frac{n-5}{4n(n-1)}\\
&= \frac{p}{2n}\left(\frac{3p-2p^2}{(1-p)^2} - \left(1 -
  n\frac{p^2}{1-p}\right)\frac{n-5}{2(n-1)}\right).
\end{align*}
Choosing $p=c/n$ for any positive $c$, the expression in the parentheses converges to $-1/2$
as $n\rightarrow \infty$. Therefore it is negative for large enough
$n$,  which is what we needed.

Now, let us fix some $c >4$. We have shown above that there exists a $n_0$ such that, for any $n \ge n_0$, if $p= \frac{c}{n}$ then 
$C_0(\xalt) < C_0(\xeq)$. This shows that for $n \ge n_0$ the size of the neighborhood of $p=0$ where $\xeq$ is optimal is at most $c/n$.
On the other hand, for $n < n_0$, trivially the size of such neighborhood is at most $1 < n_0/n$. Hence, for any $n$, such size is at most $c_0/n$
with $c_0 = \max(c,n_0)$.

%
\section{Proof of Proposition~\ref{prop:equidistlargep}}
\label{ap:propequidistlarge}

The result is trivial for $n\leq 3$, so we assume in the sequel that $n>3$. We compare the cost of the single cluster with another candidate with
three clusters as follows.\\
  \begin{center}
  \begin{tikzpicture}
    \draw[|-|] (0,0)--(8,0);
    \draw[fill] (2,0) circle [radius = 0.05];
    \draw[fill] (4,0) circle [radius = 0.05];
    \draw[fill] (6,0) circle [radius = 0.05];

    \node [below] at (2,-.1) {$k$};
    \node [below] at (4,-.1) {$n-2k$};
    \node [below] at (6,-.1) {$k$};

    \node [above] at (0,.1) {$0$};
    \node [above] at (2,.1) {$1/4$};
    \node [above] at (4,.1) {$1/2$};
    \node [above] at (6,.1) {$3/4$};
    \node [above] at (8,.1) {$1$};
  \end{tikzpicture}
  \end{center}
The numbers below the dots indicate the number of sensors aggregated
at that point, $k$ will be chosen later. If we show that this new placement is better than the
single cluster for a certain $p$, it implies that having a single
cluster is not optimal. For the single cluster, the cost is always $1/2$. For the three
clusters we get

{\centering
\begin{tabular}{cl}
  1/4 & if the left and right clusters are active,\\
  1/2 & if the left and/or right cluster fails, but the middle cluster
  is active,\\
  3/4 & if the left or right and the middle cluster fails.\\
\end{tabular}\\}
We get less than $1/2$ in expectation if the probability of getting $1/4$ is higher than getting
$3/4$. The relation needed for the probabilities is
$$(1-p^k)^2 > 2p^k p^{n-2k}(1-p^k).$$
Multiplying by $p^k/(1-p^k)$ this is equivalent to
\begin{equation}
 p^k(1-p^{k}) > 2p^n.\label{eq:highpcomparison}
\end{equation}
We need to confirm this inequality with an appropriate choice of $k$.
If $p\le 1/3$, then~\eqref{eq:highpcomparison} holds with $k=1$
(and $n>3$). Otherwise, observe that 
\begin{equation}
2p^n < 2\left(1-\frac{3}{n}\right)^n < 2e^{-3} <\frac{3}{16}.\label{eq:highpcomp1}
\end{equation}
We have to choose $p^k$ from the sequence
$p,p^2,\ldots,p^{\lfloor n/2\rfloor}$. This sequence starts at $p> 1/3$ and ends at
$p^{\lfloor n/2\rfloor} < 1/2$, and the ratio of consequent elements is greater than
$1/3$. Therefore there is an element $p^k$ in the interval $(1/4, 3/4)$.
The left hand side of~\eqref{eq:highpcomparison} is a quadratic function
in $p^k$ so it is easy to verify that
$$p^k \in \left(\frac{1}{4}, \frac{3}{4}\right) \quad \Longrightarrow
\quad p^k(1-p^{k}) > \frac{3}{16}.$$
Combining this with~\eqref{eq:highpcomp1} we arrive at~\eqref{eq:highpcomparison}, which completes our proof.

\section{Proof of Lemma~\ref{lemma:bounds-EC0xrandm}}
\label{ap:boundsEC0}

We start by proving the lower bound.
The case $m=0$ is true, since in this case the cost is 1, and $H_1 = 1$ so that $\frac{H_1}{2}<1$.
Then consider $m \ge 1$.
From Lemma~\ref{lemma:EC0randm-rope} we obtain the following lower bound
\[  \Pr (C_0(\xrandm)>v) \ge  
\Pr \left(\bigcup_{i=1}^{m+1} \left\{\frac{V_i}{2} > v \right\} \right)\,.\]
Using inclusion-exclusion principle and applying Lemma~\ref{lem:cuttingrope} with
$c_1 = \dots = c_{r} = 2v$, we obtain
\[ \Pr (C_0(\xrandm)>v) \ge 
\sum_{1\le r\le m+1 \text{ s.t. } 2rv < 1} (-1)^{r-1} \binom{m+1}{r} (1-2rv)^{m} \,.\]
Then, substituting this in~\eqref{eq:lemma-random-b}, we get
\[ \E C_0(\xrandm) \ge 
\sum_{1\le r \le m+1 } (-1)^{r-1} \binom{m+1}{r} \int_0^{\frac{1}{2r}}  (1-2rv)^m \mathrm{d}v  \,.\]
By computing
$\int_0^{\frac{1}{2r}}  (1-2rv)^{m} \mathrm{d}v = \frac{1}{2r(m+1)}$ and recalling that
$\sum_{1\le r \le m +1} (-1)^{r-1} \binom{m+1}{r} \frac{1}{r} = H_{m+1}$,
we end the proof of the lower bound.

For the upper bound we proceed similarly.
By Lemma~\ref{lemma:EC0randm-rope} and the union bound, we get
\[ \Pr(C_0(\xrandm)>v) \le 
\Pr (V_1 > v)+\Pr (V_{m+1} > v)+  \Pr \left(\bigcup_{2\le i \le m} \{\tfrac{V_i}{2} > v\}\right) \,,\]
and then, by Lemma~\ref{lem:cuttingrope} 
\[\Pr (V_1 > v) = \Pr (V_{m+1} > v) =  (1-v)^{m} \]
and by the same lemma together with inclusion-exclusion principle, 
\[ \Pr \left(\bigcup_{2\le i \le m} \{\tfrac{V_i}{2} > v\} \right)
= \sum_{1\le r\le m-1 \text{ s.t. }  2rv \le 1} (-1)^{r-1} \binom{m-1}{r} (1-2rv)^m\,.\]
From this and using~\eqref{eq:lemma-random-b}, we get
\begin{align*}
\E C_0(\xrandm)
& \le 2 \int_0^1 (1-v)^{m} \mathrm{d}v
	+  \sum_{r=1}^{m-1} (-1)^{r-1} \binom{m-1}{r}  \int_{0}^{\frac{1}{2r}} (1-2rv)^m \mathrm{d}v \\
& = 2 \frac{1}{m+1} + \sum_{r=1}^{m-1} (-1)^{r-1} \binom{m-1}{r} \frac{1}{2(m+1)r} \\
& = \frac{2}{m+1} + \frac{H_{m-1}}{2(m+1)} \,,
\end{align*}
which proves the upper bound.

\section{Proof of Lemma~\ref{lemma:bounds-Cxrand}}
\label{ap:boundsCXrand}

To get the lower bound, we consider~\eqref{eq:cost-xrandm}. By discarding terms with large $m$ and using Lemma~\ref{lemma:bounds-EC0xrandm}, we get
\begin{align*}
\E C(\xrand) &\ge 
\sum_{m=0}^{\lceil(1-p+\eps) n \rceil-1} \Pr(|A|=m)  \frac{H_{m+1}}{2(m+1)} \\ &\ge
\Pr(|A| < \lceil(1-p+\eps) n \rceil) 
	\min_{m < \lceil(1-p+\eps) n \rceil} \frac{H_{m+1}}{2(m+1)}
\end{align*}
It is easy to show
that $\frac{H_m}{m}$ is decreasing with $m$, so that 
\[ \min_{m < \lceil(1-p+\eps) n \rceil} \frac{H_{m+1}}{2(m+1)} = \frac{H_{\lceil(1-p+\eps) n \rceil}}{\lceil(1-p+\eps) n \rceil} \,.\]
Then,
\[\Pr(|A| < \lceil(1-p+\eps) n \rceil) =
1 - \sum_{\lceil(1-p+\eps) n \rceil}^n \binom{n}{m} (1-p)^m p^{n-m} =
1 - \sum_{m'=0}^{\lfloor(p-\eps) n \rfloor} \binom{n}{m'} p^{m'} (1-p)^{n-m'} \]
and,
 by Hoeffding inequality,  
 $\sum_{m'=0}^{\lfloor (p-\eps) n \rfloor} \binom{n}{m'} p^{m'} (1-p)^{n-m'} \le e^{-2 \eps^2 n}$,
 which ends the proof of the lower bound.

For the upper bound, we proceed similarly.
From now on, we assume that $\lfloor (1-p-\eps) n \rfloor \ge 2$; notice that the bound is trivially true otherwise.
We consider~\eqref{eq:cost-xrandm} and we split the summation in two terms:
a first term with $m \le \lfloor (1-p-\eps) n \rfloor$, in which we use the trivial bound $C_0(\xrandm) \le 1$,
and the remaining sum in which we use the upper bound from Lemma~\ref{lemma:bounds-EC0xrandm}, as follows
\[ \E C(\xrand) \le
\Pr\left(|A| \le \lfloor(1-p-\eps) n \rfloor\right)
+ \Pr\left(|A| > \lfloor(1-p-\eps) n \rfloor\right) 
		\max_{m > \lfloor (1-p-\eps) n \rfloor} \frac{4+H_{m-1}}{2(m+1)} \,. \]
By Hoeffding inequality, $\Pr\left(|A| \le \lfloor(1-p-\eps) n \rfloor\right) \le e^{-2 \eps^2 n}$.
For the second term,
it is easy to show that 
$ \frac{H_{m-1}+4}{m+1} $ is decreasing with $m$
and hence
\[\max_{m > \lfloor (1-p-\eps) n \rfloor} \frac{4+H_{m-1}}{2(m+1)}
= \frac{H_{\lfloor (1-p-\eps)n \rfloor}+4}{2(\lfloor (1-p-\eps)n \rfloor+2)}\,.\]
Finally we use the trivial bound $\Pr\left(|A| > \lfloor(1-p-\eps) n \rfloor\right)  \le 1$.

Then, the formulation of the upper bound stated in the proposition, which is slightly weaker
but has the advantage of not explicitly requiring to assume $\lfloor (1-p-\eps) n \rfloor \ge 2$, is obtained since
$ H_{\lfloor (1-p-\eps)n \rfloor} \le  H_{\lceil (1-p-\eps)n \rceil}$ and in the denominator $2\lfloor (1-p-\eps)n \rfloor +4 \ge 2 \lceil (1-p-\eps)n \rceil + 2$.

\bibliographystyle{ieeetran}


\end{document}